\documentclass[ a4paper,11pt]{article}
\usepackage{amsmath,amssymb,amsfonts,latexsym, amsthm}
\usepackage[active]{srcltx}

\newtheorem{teor}{\bf Theorem}[section]

\newtheorem{lemma}[teor]{\bf Lemma}

\newtheorem{cor}[teor]{\bf Corollary}

\newtheorem{defn}{Definition}[section]
\newtheorem{rem}{Remark}[section]

\newcommand{\Rk}{{\mathbb R}^{k}}
\newcommand{\Rnk}{{\mathbb R}^{N-k}}
\newcommand{\Rn}{{\mathbb R}^{N}}

\newcommand{\fa} {\forall}

\newcommand{\pa} {\partial}

\newcommand{\al} {\alpha}
\newcommand{\ba} {\beta}
\newcommand{\de} {\delta}

\newcommand{\Ga} {\Gamma}
\newcommand{\Ome} {\Omega}

\newcommand{\la} {\lambda}

\newcommand{\no} {\nonumber}
\newcommand{\noi} {\noindent}

\newcommand{\na} {\nabla}
\newcommand{\va} {\varphi}
\newcommand{\var} {\varepsilon}

\newcommand{\spa}{\vspace{.2in}}
\newcommand{\df}{\stackrel{def}{=}}
\newcommand{\Onl}{\Omega^\nu_\lambda}
\newcommand{\Onm}{\Omega^\nu_\mu}
\newcommand{\lan}{\langle}
\newcommand{\ran}{\rangle}

\newcommand{\R}{\mathbb{R}}

\newcommand{\La}{\Lambda}

\newcommand{\Hunoz}{H^{1}_{0}(\Omega)}
\newcommand{\Wuno}{W^{1,p}_{0}(\Omega)}
\newcommand{\ps}{p^*(s)}
\newcommand{\pt}{p^*(t)}
\newcommand{\Iom}{\int_{\Omega}}

\newcommand{\Idom}{\int_{\partial \Omega}}

\newcommand{\un}{u_{n}}

\newcommand{\deb}{\rightharpoonup}

\newcommand{\unl}{u_\lambda^\nu}
\newcommand{\ynl}{y_\lambda^\nu}

\date{}
\begin{document}
\begin{center}
{{\textbf{\large  HARDY-SOBOLEV TYPE EQUATIONS\\ \ \\  FOR $p$-LAPLACIAN, $1<p<2$\\ \ \\ IN BOUNDED DOMAIN}}}
 \end{center}

\vspace{.2in}
\begin{center}
{M. Bhakta and  A. Biswas} 
\end{center}

\begin{center}
 {\small TIFR Centre for Applicable Mathematics, Post Bag No. 6503\\
 Sharadanagar,Chikkabommasandra, Bangalore 560 065.\\ E-mail:  {\it
 mousomi@math.tifrbng.res.in},
{\it anup@math.tifrbng.res.in}}
\end{center}

\begin{abstract}
We study quasilinear degenerate singular elliptic equation of type $-\Delta_p u = \frac{u^{p^*(s)-1}}{|y|^t}$
 in a smooth bounded domain $\Ome$ in $\Rn=\Rk\times\R^{N-k}$, $x=(y,z)\in\Rk\times\R^{N-k}$, $2 \leq k<N$ and $N\geq 3$, $1<p<2$, $0\leq s\leq p$,\ $0\leq t\leq s$ and $p^*(s)=\frac{p(n-s)}{n-p}$.
We study existence of solution for $t<s$, non-existence in a star-shaped domain for $t=s$  and $s<k\big(\frac{p-1}{p}\big)$. We also show that solution 
 is in $C^{1,\al}(\Ome)$ for  some $0<\al<1$ provided $t<\frac{k}{N}\big(\frac{p-1}{p}\big)$. The regularity of solution
can be improved to the class $W^{2,p}(\Ome)$ when $t<k(\frac{p-1}{p})$.
We also study some property of the singular sets in a cylindrically symmetric domain using the method of symmetry.
\end{abstract}

\noindent {\bf 2000 Mathematics Subject Classification}: 35J60, 35J70, 35J75, 35B09, 35B65.

\section{Introduction}
In this article we study degenerate quasilinear singular elliptic equation of the type

\begin{equation}\label{A}
\left.\begin{array}{rlllll}
 -\Delta_p u &=& \frac{u^{p^*(s)-1}}{|y|^t} \ & \mbox {in}& \Ome\\
u &\geq& 0 \ \ \ &\mbox{in}& \Ome\\
u &\in& \ W^{1,p}_0(\Ome),
\end{array}\right\} 
\end{equation}

where $\Delta_p$ denotes the $p$-laplacian operator, $\Delta_p u=\mbox{div}(|\na u|^{p-2}
\na u)$ and
 $\Ome$ is a smooth bounded domain in $\Rn=\Rk\times \Rnk$, $x=(y,z)\in\Rk\times\R^{N-k}$, 
 $2 \leq k<N$ and $N\geq 3$, $1<p<2$, $0\leq s<p$,\ $0\leq t\leq s$
 and $p^*(s)=\frac{p(N-s)}{N-p}$.

\spa

 By a non-trivial solution of (\ref{A}) we
mean $0\not\equiv u\in\Wuno$ satisfying
\begin{displaymath}
 \Iom |\na u|^{p-2}\na u\na v =\Iom \frac{u^{\ps-1}}{|y|^t}v\ \ \fa \ v\in\Wuno.
\end{displaymath}
Equivalently, $u$ is a critical point of the functional $I$ given by
\begin{displaymath}
 I(u)\df \frac{1}{p}\Iom |\na u|^p -\frac{1}{\ps}\Iom \frac{|u|^{\ps}}{|y|^t}, \ \ u\in\Wuno.
\end{displaymath}
$I$ is a well defined $C^1$ functional on $\Wuno$ for any open subset of $\Rn$ due to the 
following Hardy-Sobolev-Inequality:
\begin{displaymath}
 \int_{\Rn}\frac{|u|^{\ps}}{|y|^s}\leq C(\int_{\Rn}|\na u|^p)^{\frac{N-s}{N-p}}
\ \ \fa \ u\in C^\infty_0(\Rn),
\end{displaymath}
where $C$ is a constant depending on $s,\ p,\ N, k$ (see \cite{BT}, \cite{FMT}). Clearly, the limiting case $s=0$
corresponds to classical Sobolev inequality and for $s=p$ this inequality still holds
which is known as Hardy's Inequality. It can be easily checked that if $\Ome$ is a 
bounded domain in $\Rn$ then for $t\leq s$
\begin{eqnarray}
\Iom \frac{|u|^{\ps}}{|y|^t}\leq C(\Iom |\na u|^p)^{\frac{N-s}{N-p}}\ \ \fa \ u\in\Wuno
\label{HSI}.
\end{eqnarray}
When $\Ome=\Rn$, $t=s$, the existence of critical points of $I$ has been studied in \cite{BT}. In a limiting case $p=2$, existence and classification of solutions are
exclusively studied in \cite{MS}, \cite{FMS}. In this set up, uniqueness of solution in $D^{1,2}_{loc}(\Rn)$ has been studied in \cite{CL}. When $p=2$ and 
$\Ome$ bounded, (\ref{A}) with $t=s$ has been studied in \cite{BS}.
But in case of bounded domain we can not
say the solution exists in general because when $t=s$, (\ref{A}) turns out to be with
critical exponent. For this case we prove the non-existence result in Theorem 4.1 but
the difficulties here is that the solution is not regular enough to justify the 
calculations of usual Pohozaev type identity. We can get $u$ in $C^1(\bar\Ome\setminus
\{y=0\})$ due to the singularity of equation at $\{y=0\}$. Therefore we study $W^{2,p}$
regularity properties of solution in Section 3 and extend the same regularity up to
the boundary if $\pa\Ome$ is orthogonal to $\{y=0\}$ (see Section 3 for definition) (for $t<k(\frac{p-1}{p})$) which helps us to prove the non existence result
in star-shaped domain in section 4.

\spa

In section 5, we study symmetry properties of solution and its relation with the set $\{Du=0\}$.
A domain $\Ome$ is said to be cylindrically symmetric about $x_0=(y_0,z_0)\in\Ome$
if the following two conditions hold:
\begin{itemize}
 \item $\Ome$ is symmetric in $y\in\Rk$ about the point $(y_0,z)\in\Ome$ for any 
arbitrarily fixed $z$.
\item $\Ome$ is symmetric in $z\in\Rnk$ about the point $(y,z_0)\in\Ome$ for any arbitrarily fixed $y$.
\end{itemize}

 We say a function $u$ is 
symmetric in variable y if there exists $y_0\in\Rk$ such that 
\begin{itemize}
 \item for any choice of $z\in\Rnk$, $u(\cdot,z)$ is symmetric non-increasing about $(y_0,z)$
in $\Rk$
\end{itemize}
and symmetric in variable z if there exists $z_0\in\Rnk$ such that 
\begin{itemize}
\item for any choice of $y\in\Rk$, $u(y,\cdot)$ is symmetric non-increasing about $(y,z_0)$
in $\Rnk$.
\end{itemize}
We say $u$ is cylindrically symmetric about $(y_0,z_0)$ if the two conditions above hold. 
In  the case $p=2$, several results have been
obtained starting with the famous paper \cite{GNN} by Gidas, Ni and Nirenbarg. Let $\Ome$ be a bounded domain in $\Rn$, $N\geq 2$, which is convex and symmetric in the $x_1$ direction and consider the problem
\begin{equation}\label{B}
\left.\begin{array}{rlllll}
 -\Delta_p u &=& f(u) \ & \mbox {in}& \Ome\\
u &>& 0 \ \ \ &\mbox{in}& \Ome\\
u &=& 0 \  \ \ &\mbox{on}& \pa\Ome,
\end{array}\right\} 
\end{equation}
In this paper they used moving plane method to prove (among other results) that if $p=2$, every classical solution to (\ref{B}) is symmetric with respect to the hyperplane $T_0=\{x=(x_1,x^\prime)\in\Rn : x_1=0\}$ and strictly increasing in $x_1$ for $x_1<0$, provided $\Ome$ is smooth and $f$ is Lipschitz continuous.As a corollary if $\Ome$ is a ball, $s=0=t$, then $u$ is radially symmetric and strictly radially decreasing.
One of the several reason that the paper had a big impact was it brought to attention the moving plane method which since then has been largely used in many different problems. This
method was essentially based on maximum principle which was first used by Serrin
\cite{S}. For $p=2$ and $\Ome=\Rn$, cylindrical symmetry of solution of (\ref{A})
was shown in \cite{FMS} (see also \cite{GM}). The difficulties in extending the result for the case $p\neq 2$
is to overcome the hurdle of extending the properties of solutions of strictly elliptic equation to solutions of $p$-laplacian equation. In 
particular, comparison principles used for strictly elliptic operators are not
available for degenerate operator considered. 

\spa

A first step towards extending the moving plane method to the solutions of problem involving p-laplacian operator has been done in \cite{LD}. Later for $1<p<2$, $N\geq 2$, if $u\in\Wuno\cap C^1(\bar\Ome)$ solves (\ref{B}) where $f$ is lipschitz continuous, symmetry of solution using moving plane method is done in \cite{DP}.

\spa

But none of those methods are applicable for our set up as we have both degeneracy and partial singularity in $f$. The degeneracy does not allow us to use the
maximum principle for strictly elliptic operators and singularity in $\{y=0\}$ does not allow us to use the moving plane method as in \cite{DP}
(Without much effort one can apply the method in \cite{DP} if the partial singularity $\{y=0\}$ is replaced by singularity at $\{x=0\}$). 
As we mentioned earlier, the key ingredient in the method of moving plane is weak comparison principle. In the presence of
one point singularity $\{x=0\}$ one can modify the weak comparison principle in \cite{DP} to get a suitable comparison principle but
in the presence of partial singularity it seems difficult to get a modified version. However we can get a weaker comparison principle (see Theorem
\ref{WCP}). Therefore getting the cylindrical symmetry of $u$ satisfying (\ref{A}) in cylindrically symmetric domain is quite challenging. 
However, one can study some properties of the set of degeneracy $Z=\{x\in\Ome : Du=0\}$ in a cylindrically symmetric domain.  
Connection of $Z$ with the symmetry has also been studied in \cite{LD}. Let
$\mathcal{H}^k$ denote the $k$-dimensional Hausdroff measure. Then it is easy to observe that
\begin{displaymath}
u\ \mbox{is symmetric in variable}\ y\Rightarrow \mbox{either}\ Z\subset\{y=0\}\  \mbox{or}\ \mathcal{H}^{k-1}(Z)>0.
\end{displaymath}
In section 5, we show the following is true
\begin{displaymath}
\mathcal{H}^{k-1}(Z)=0\Rightarrow u\ \mbox{is symmetric in variable}\ y\ \mbox{and}\ Z\subset\{y=0\}.
\end{displaymath}
In other words, we show that in a cylindrically symmetric domain the set of degeneracy $Z$ of $u$ satisfying (\ref{A}) can not be "small" (in measure
theoretic sense) unless $Z=\{0\}$.
\spa

\section{ Existence of non trivial solutions in subcritical case}

\begin{teor}
 There exists a non trivial solution of the equation (\ref{A}) when $t<s$. Moreover if
$t<\frac{kp}{N}$ and $\Ome$ is connected then the non trivial solution is strictly positive in $\Ome$.
\end{teor}

\begin{proof}
Let us define,
\begin{displaymath}
 I:\Wuno\to\R \  \  \ \mbox{by}
\end{displaymath}
\begin{displaymath}{\normalsize {\large {\Large }}}
 I(u)\df \frac{1}{p}\Iom |\na u|^p-\frac{1}{\ps}\Iom\frac{|u|^{\ps}}{|y|^t}
\end{displaymath}
First we check that $I$ satisfy all the conditions of Mountain Pass Theorem (see \cite{AR}). $I$ is well defined $C^1$ functional on $\Wuno$ due to (\ref{HSI}).\\
Clearly (i) $I(0)=0$ is satisfied.\\[1mm]
(ii) Let us take $u\in\Wuno$ s.t. $\|u\|_{\Wuno}=r$, where $r$ will be chosen later.
 Then
\begin{displaymath}
 \Iom\frac{|u|^{\ps}}{|y|^t}\leq\int_{\{|y|\leq
 1\}\cap\Ome}\frac{|u|^{\ps}}{|y|^s}+\int_{\{|y|>1\}\cap\Ome}|u|^{\ps}\leq C \|u\|^{\ps}_{\Wuno}
\end{displaymath}
Therefore $$ I(u)\geq\frac{r^p}{p}-C\frac{r^{\ps}}{\ps}.$$
Since $p<p^*(s)$, we have $I(u)>0$ for $r>0$ small enough.\\[1mm]
(iii) Let us fix $u\in\Wuno$ s.t $u\not\equiv 0$. Now for $\bar t>0$
\begin{displaymath}
 I(\bar tu)=\frac{\bar t^p}{p}\Iom|\na u|^p-\frac{\bar
 t^{\ps}}{\ps}\Iom\frac{|u|^{\ps}}{|y|^t}
\end{displaymath}
So, for $\bar t$ to be large enough we have $I(\bar tu)<0$.\\
(iv) Suppose $\{\un\}\subset\Wuno$ s.t. $\un\geq 0$ and
\begin{displaymath}
 \sup_{n} I(\un)\leq C,\ \  I^\prime(\un)\to 0\  \  \mbox{in}\ \  W^{-1,p^\prime}(\Ome).
\end{displaymath}
Claim: $\{\un\}$ has a convergent subsequence in $\Wuno$.\\
To prove the claim first we note that $\un$ is bounded in $\Wuno$ (by standard argument). More preciously, since $\langle I^\prime(\un),\un\rangle=o(1)\|\un\|$, computing $I(\un)-\frac{1}{\ps}\langle I^\prime(\un),\un\rangle$, we get $\|\un\|^p_{\Wuno}\leq M+o(1)\|\un\|_{\Wuno}$ and hence boundedness follows.  Therefore passing to a subsequence we may assume that $\un\deb u$ in $\Wuno$, $\un\to u$ in $L^q(\Ome)$ for $q<p^*$.\\[1mm]
For $v\in C^\infty_0(\Ome)$ we have
\begin{equation}\label{a}
o(1)=\langle I^\prime(\un),v\rangle=\Iom |\na\un|^{p-2}\na\un\na
 v-\Iom\frac{|\un|^{\ps-2}\un v}{|y|^t}
\end{equation}
where $o(1)\to 0$ in $W^{-1, p^\prime}(\Ome)$ as $n\to\infty$. Using Vitali's Theorem we have 
\begin{equation*}
\frac{|\un|^{\ps-2}\un}{|y|^t}\to \frac{|u|^{\ps-2}u}{|y|^t} \   \  \mbox{in}\ \  L^1(\Ome)
\end{equation*}
and therefore from Boccardo-Murat (see \cite{BM}) we get
\begin{equation*}
\na\un\to\na u \ \ \mbox{a.e. in} \ \ \Ome.
\end{equation*}
Hence RHS of \ref{a} converges to $\Iom |\na u|^{p-2}\na u\na v-\Iom\frac{|u|^{\ps-2}uv}{|y|^t}=\langle I^\prime(u),v\rangle$.
Therefore we have $I^\prime(u)=0$.\\
Now
\begin{displaymath}
 o(1)=\langle I^\prime(\un)-I^\prime(u),\un-u\rangle
\end{displaymath}
\begin{displaymath}
 =\Iom(|\na\un|^{p-2}\na\un-|\na u|^{p-2}\na
 u)\na(\un-u)-\Iom\frac{|\un|^{\ps-2}\un-|u|^{\ps-2}u}{|y|^t}(\un-u)
\end{displaymath}
Therefore using the inequality $\langle |a|^{p-2}a-|b|^{p-2}b, a-b\rangle \geq C(|a|+|b|)^{p-2}|a-b|^2$ for $a, b \in \Rn$, we have 
\begin{eqnarray}
 \Iom(|\na\un|+|\na u|)^{p-2}|\na(\un-u)|^2\leq
 C\Iom\frac{|\un|^{\ps-2}\un-|u|^{\ps-2}u}{|y|^t}(\un-u)+ o(1).\label{eqn}
\end{eqnarray}
We already know that $\un\to u$ a.e. in $\Ome$ and
\begin{eqnarray*}
 \int_{\Ome_1}\frac{|\un|^{p^*(s)-1}|u|}{|y|^t} &\leq&
\Big(\int_{\Ome_1}\frac{|\un|^{p^*(s)}}{|y|^t}\Big)^\frac{p^*(s)-1}{p^*(s)}
\Big(\int_{\Ome_1}\frac{|u|^{\ps}}{|y|^t}\Big)^\frac{1}{\ps}
\\
\int_{\Ome_1}\frac{|u|^{\ps-1}|\un|}{|y|^t} &\leq &
\Big(\int_{\Ome_1}\frac{|u|^{p^*(s)}}{|y|^t}\Big)^\frac{p^*(s)-1}{p^*(s)}
\Big(\int_{\Ome_1}\frac{|\un|^{\ps}}{|y|^t}\Big)^\frac{1}{\ps},
\end{eqnarray*}
for any $\Ome_1\subset \Ome$. Therefore using (\ref{HSI}) and Vitali's Theorem  we have
\begin{eqnarray*}
 \int_{\Ome}\frac{|\un|^{p^*(s)-2}\un u}{|y|^t} &\to& \int_{\Ome}\frac{|u|^{\ps}}{|y|^t}
\\
\int_{\Ome}\frac{|u|^{p^*(s)-2}u \un}{|y|^t} &\to& \int_{\Ome}\frac{|u|^{\ps}}{|y|^t}.
\end{eqnarray*}
Again from (H-S) inequality we have
\begin{displaymath}
 \Big(\int_{\Ome}\frac{|u|^q}{|y|^t}\Big)^\frac{1}{q}\leq \Big(\int_\Ome|\na
u|^p\Big)^\frac{1}{p} \ \ \forall \ q\leq p^*(t)\ \mbox{and}\ u\in\Wuno.
\end{displaymath}
Therefore the map $f:\Wuno\to L^q(\Ome)$ defined by
\begin{displaymath}
 f(u)\df \frac{u}{|y|^\frac{t}{q}}
\end{displaymath}
is continuous for $q\leq p^*(t)$ and compact for $q<p^*(t)$. Since $t<s$, we have $p^*(t)
>\ps$ and hence
\begin{displaymath}
 \int_\Ome\frac{|\un|^{\ps}}{|y|^t}\to \int_\Ome\frac{|u|^{\ps}}{|y|^t}.
\end{displaymath}
Therefore from (\ref{eqn})
we have $\un\to u$ in $\Hunoz$. Since $p<2$, we have $\un\to u$ in $\Wuno$.\\
We already have, $\displaystyle\inf_{\|u\|=r}I(u)=b$ for some $b>0$. Now applying Mountain Pass Theorem  we get $\ba$ is a critical value of $I$, where
\begin{displaymath}
 \ba=\inf_{g\in\Ga}\max_{0\leq t\leq 1}I(g(t))\geq b>0,
\end{displaymath}
which implies $\ba>0$ and $\exists\ u\in\Wuno$ s.t. $I(u)=\ba, I^\prime(u)=0$. Now $I(u)>0$ implies $u\not\equiv 0$. Since $I(u)=I(|u|)$, we can get $|u|$ as a critical point of I as well and this proves existence of non trivial solution to (\ref{A}). In fact, $u\in L^{\infty}(\Ome)$ as we would prove in the next section. If $t<\frac{kp}{N}$, then the solution $u$
of (\ref{A}) is H\"older continuous, using Theorem $7.3.1$ in \cite{SP}. Since $\Ome$ is connected, using the
strong maximum principle (Theorem $2.2$ in \cite{LD}) we have the non-trivial 
solution $u>0$ in $\Ome$.
\end{proof}

\spa

\section{Regularity of solution} 

We have already seen existence of non-trivial solution in case of $t<s$, but when $t=s$
solution of (\ref{A}) may not exist in all bounded domain in general as we will see in the next section. But whenever the solution exists we can have interior regularity as well.
To show $C^{1,\al}_{loc}$ regularity we'll use the following two theorems:\\

\begin{teor}\label{Vas}
Let $\Ome$ be an open subset of $\Rn$, $1<p<N$, $0\leq s\leq p$, $s<k$ and
$s(N-k)<k(N-p)$. Let $u\in D^{1,p}(\Ome)$ be a non-negative weak solution of
the inequality
\begin{displaymath}
 -div(|\na u|^{p-2}\na u)\leq V\frac{|u|^{p-2}u}{|y|^s}\ \ \mbox{in}\ \Ome.
\end{displaymath}
\item[(a)] If $V\in L^{r^\prime}(\Ome)$, $r^\prime=\frac{p^*}{\ps-p}$, then $u\in 
L^q(\frac{dx}{|y|^t})$ for any $0\leq t<\min\{p,s\}$ and $q\geq\ps$. In particular
$u\in L^q(\Ome)$ for every $p^*\leq q<\infty$.
\item[(b)] If $V\in L^{t_0}(\Ome)\cap L^{r^\prime}(\Ome)$ for some $t_0>r^\prime$,
then $u\in L^\infty(\Ome)$. 
\end{teor}
For the proof we refer the reader to Theorem 2.1 in \cite{DV}.

\spa

\begin{teor}\label{ED}
Let $u\in W^{1,p}_{loc}(\Ome)$ be a weak solution of
\begin{displaymath}
 -\mbox{div}(|\na u|^{p-2}\na u)=\varphi,\ p>1; \ \varphi\in L^q_{loc}(\Ome)\
\mbox{for some}\ q
>\frac{Np}{p-1}
\end{displaymath}
then $u\in C^{1,\al}_{loc}(\Ome)$.
\end{teor}
For the proof of above theorem we refer \cite{ED}.

\spa

\begin{defn}
Let $\Ome$ is an open subset of $\Rn$ with smooth boundary. We say that $\pa\Ome$ is orthogonal to the singular set if for every $(0, z_0)\in \pa\Ome$ the normal at $(0, z_0)$ is in $\{0\}\times R^{N-k}$
\end{defn}

\begin{teor}\label{C1al}
If $u$ is a solution of Equation (\ref{A}) in $\Ome$ with $\pa\Ome$ is orthogonal to the singular set, then $u\in C^1(\bar\Ome\setminus\{y=0\})$ and for $t<\frac{k}{N}(\frac{p-1}{p})$, $u\in C^{1,\al}(\bar\Ome)$ for some $0<\al<1$.
\end{teor}

\begin{proof}
 From Equation (\ref{A}) we have 
\begin{displaymath}
 -\mbox{div}(|\na u|^{p-2}\na u)=V\frac{u^{p-1}}{|y|^t}
\end{displaymath}
where $V=u^{\ps-p}$. Since we have $t\leq p<2\leq k$, so $t<k\frac{N-p}{N-k}$. Therefore to claim $u\in L^q(\Ome)\ \  \fa\ \ p^*\leq q<\infty$, we need to check $V\in L^r(\Ome)$ where $r=\frac{p^*}{\pt-p}$ [using Theorem \ref{Vas}]. Now, $t\leq s\Rightarrow \frac{p^*}{\pt-p}\leq \frac{p^*}{\ps-p}$. Clearly $V\in L^{\frac{p^*}{\ps-p}}(\Ome)$, therefore $V\in L^r(\Ome)$. Therefore $u\in L^q(\Ome)\ \  \fa\ \ p^*\leq q<\infty$.
Now let us choose $t_0>0$ s.t. $(\ps-p)t_0>p^*$. Therefore $V\in L^{t_0}(\Ome)$ with $t_0>r$ and hence $u\in L^\infty(\Ome)$ using Theorem \ref{Vas}.

\spa

We note that for any domain $\Ome_1\subset\Ome$ such that $\bar\Ome_1\cap\{y=0\}=\emptyset$,
$-\Delta_p u\in L^\infty(\Ome_1)$ as $u\in L^\infty(\Ome)$. So we can apply Theorem \ref{ED}
in $\Ome_1$ and therefore $u\in C^1(\Ome_1)$. For the boundary points away from $\{y=0\}$, again
we have the right hand side of (\ref{A}) in $L^\infty$ and hence we can use the $C^{1,\al}$
estimate by Lieberman \cite{L} in a half ball to conclude that $u$ is $C^1$ at those points. This completes the proof of the first part.

\spa

If $t<\frac{k}{N}\big(\frac{p-1}{p}\big)$, we can choose $q>0$ s.t. $\frac{Np}{p-1}<q<\frac{k}{t}$. Since $u\in L^\infty(\Ome)$  we have right hand side of Equation (\ref{A}) in $L^q(\Ome)$. Hence using Theorem \ref{ED} we have $u\in C^{1,\al}_{loc}(\Ome)$ for some $0<\al<1$. If $x\in\pa\Ome\setminus\{y=0\}$, 
we can have a ball $B(x,r)\subset\Rn$ such that $\{y=0\}\cap(B(x,r)\cap\Ome)=\emptyset$.
Therefore $C^{1,\al}$ regularity around $x$ follows from Lieberman (cf. \cite{L}).

\spa

Therefore it remains to prove the boundary regularity at $\pa\Ome\cap\{y=0\}$. Since $\pa\Ome$ is orthogonal to the singular set, normal at $x_0$ is in ${0} \times R^{N-k}$ . We may assume normal at $(0, z_0 )$ is $(0, \ldots, 1)$ (rotating in the z variable if needed). Now at $x_0$ we use a local reßection method as in \cite{T}. Set $B_r (a) = \{x \in \Rn : |x-a| < r\}$. For $ x_0 \in \pa\Ome \cap \{y = 0\}$, there exists $R > 0$ and a smooth function $f$ such that $B_R(x_0 ) \cap \Ome = \{ (y, z) : z_{N-k} > f (y, z_1 , \ldots , z_{N-k-1}) \}$. Let us ßatten the boundary near $x_0$ . Therefore we can have smooth diffeomorphism of the form $\eta(y, z ) = (y, Z_1 , \ldots , Z_{N-k-1} , Z_{N-k} - f (y, Z_1 , \ldots,  Z_{N-k-1} ))$ where $Z$ denotes the coordinate after rotation. 
\begin{equation}
 \left\{\begin{array}{rll}
  \eta(B_R(x_0)\cap\Ome) &=& B_1(0)\cap\{(y, z_1, \ldots, z_{N-k}): \ z_{N-k}>0\}=B_1^+,
\\
\eta(B_R(x_0)\setminus\bar\Ome) &=& B_1(0)\cap\{(y, z_1, \ldots, z_{N-k}): \ z_{N-k}<0\}=B_1^-.
 \end{array}\right.
\end{equation}

Let $R(B_1^+)$ denote the reflection of $B_1^+$ about $\{X_N=0\}$ and $\tilde \Ome=B_1^+\cup R(B_1^+)\cup \mbox{int}  (\{X_N=0\}\cap\pa B_1^+)$. We define $v=u\circ\eta^{-1}$ and $\bar{v}$ by
\[\bar{v}(X_1,\ldots,X_N)=\left\{\begin{array}{ll}
v(X_1,\ldots,X_N)\ & \mbox{if}\ X_N>0,\\
-v(X_1,\ldots,-X_N)\ & \mbox{if}\ X_N<0.
                                \end{array}
\right.\]

As $\bar{v}$ vanishes on $X_N=0$ a straightforward computation shows that it satisfies
an equation of following type (see Theorem 2.3 in \cite{BS})
\begin{displaymath}
 -\mbox{div}A(x,D\bar{v})=\frac{\bar {v}^{\ps-1}|\mbox{det} B(x)^{-1}|}{|y|^t} \ \  \mbox{in}\  \  \tilde\Ome,
\end{displaymath}
where $A(x, \zeta)= |\zeta B(x)|^{p-2}\zeta B(x)B^T(x)|\mbox{det} B(x)^{-1}|$ and $B(x)=\nabla\eta(x) $satisfies all the structural and regularity assumptions need to apply the local regularity result of 
Theorem $2$ in \cite{ED}. Therefore we have local $C^{1,\alpha}$ regularity around $x_0$. This completes the proof.

\end{proof}

\spa

\begin{teor}\label{W2p}
 If $u$ is a weak solution of the equation (\ref{A}) then $u\in W^{2,p}(\Ome)$
for $t<\frac{k(p-1)}{p}$ if $\pa\Ome$ is orthogonal to the singular set.
\end{teor}

\begin{proof} First we prove interior regularity for the
quasilinear degenerate equation of type
\begin{eqnarray}
-\mbox{div}A(x,\na u) = \dfrac{u^{\ps-1}}{|y|^t}\label{quasi}
\end{eqnarray}
where $A(x,\zeta):\Ome\times\Rn\to\Rn$ satisfies the following properties:
\begin{eqnarray}
(A(x,\zeta)-A(x,\zeta'))(\zeta-\zeta') &\geq & C(1+|\zeta|^2+|\zeta'|^2)^\frac{p-2}{2}
|\zeta-\zeta'|^2\label{qp1}
\\
|A(x,\zeta)| &\leq & C_1|\zeta|^{p-1}\label{qp2}
\\
\sum_{i=1}^n |A_i(x,\zeta)-A_i(y,\zeta)| &\leq & C_2(1+|\zeta|)^{p-1}|x-y|.\label{qp3}
\end{eqnarray}
where $A_i$ is the $i$th coordinate of $A$.
Let $u\in\Wuno\cap L^\infty(\Ome)$ be a solution of (\ref{quasi}).
Therefore for any $v\in\Wuno$ we have
\begin{eqnarray}
\Iom A(x,\na u(x)).\na v = \Iom \frac{u^{\ps-1}v}{|y|^t}.\label{eqn2}
\end{eqnarray}
Let $R>0$ be such that $B_{4R}\subset\Ome$ where $B_R$ denotes the ball of radius
$R$ around some point in $\Ome$. Let $\xi\in C^\infty_0(\Ome)$ be such that $\mbox{supp}(\xi)\subset B_{2R}$
, $\xi=1$ in $B_R$, $0\leq\xi\leq 1$ $|\na\xi|\leq\frac{1}{R}$ and $|D^2\xi|\leq\frac{c}{R^2}$ for
some constant $c$. Now for $h>0$, we choose
\begin{displaymath}
 v(x)\df D^{-h}_i(\xi^2(x)D^h_iu(x)),
\end{displaymath}
where
\begin{displaymath}
 D^h_iu(x)\df \frac{u(x+he_i)-u(x)}{h}\ \ \mbox{where}\ i=1,\ldots,N,
\end{displaymath}
and $e_i$ is the $i$-th canonical basis vector in $\Rn$.
Therefore from (\ref{eqn2}) we have
\begin{eqnarray}
\Iom D^h_i(A(x,\na u))\na(\xi^2D^h_iu)=-\Iom\frac{u^{\ps-1}D^{-h}_i(\xi^2D^h_iu)}{|y|^t}.\label{eqn3}
\end{eqnarray}
Now
\begin{eqnarray*}
 \mbox{LHS of}\ (\ref{eqn3}) &=& \Iom \xi^2 \langle D^h_i(A(x,\na u)),\na(D^h_iu)
\rangle
\\
&& + 2\Iom \xi(D^h_iu)\langle D^h_i(A(x,\na u)),\na\xi\rangle.
\end{eqnarray*}
Hence from (\ref{eqn3}) we have
\begin{eqnarray}
\Iom \xi^2 \langle D^h_i(A(x,\na u)),D^h_i(\na u)\rangle &= & 
-2\Iom \xi(D^h_i u)\langle D^h_i(A(x,\na u)),\na\xi\rangle \nonumber
\\
&& - \Iom \frac{u^{\ps-1}D^{-h}_i(\xi^2D^h_iu)}{|y|^t}.\label{eqn4}
\end{eqnarray}
Using (\ref{qp1}) we get
\begin{eqnarray*}
 \mbox{LHS of (\ref{eqn4})} &=& \Iom\frac{\xi^2}{h^2}\Big[A(x+he_i,\na
 u(x+he_i))-A(x+he_i,\na u(x))
\\
& & +  A(x+he_i,\na u(x))-A(x,\na u(x))\Big](\na u(x+he_i)-\na u(x))
\\
&\geq & C\Iom \xi^2\Big[(1+|\na u(x+he_i)|^2+|\na u(x)|^2)^\frac{p-2}{2}|D^h_i\na u(x)|^2+
\\
&& (A(x+he_i,\na u(x))-A(x,\na u(x)))\frac{D^h_i(\na u(x))}{h}\Big]
\end{eqnarray*}

Let us define 
\begin{displaymath}
 Y(x)^2=1+|\na u(x)|^2+|\na u(x+he_i)|^2
\end{displaymath}

Therefore we have from (\ref{qp3}) and (\ref{eqn4})
\begin{eqnarray}
C\Iom {\xi}^2|D_i^h(\na u(x))|^2 Y^{p-2}&&\leq C_2\Iom{\xi}^2(1+|\na u(x)|)^{p-1}|D_i^h\na u(x)|\no\\
&&-2\Iom \xi(D^h_i u)\langle D^h_i(A(x,\na u)),\na\xi\rangle\no
\\
&& - \Iom \frac{u^{\ps-1}D^{-h}_i(\xi^2D^h_iu)}{|y|^t}.\label{a}
\end{eqnarray}

Using the following relation
for $f\in L^1_{loc}(\Ome)$
\begin{displaymath}
 D^h_i f(x)=\frac{\pa}{\pa x_i}\big(\int^1_0 f(x+the_i)dt\big),
\end{displaymath}
we have 
\begin{displaymath}
 D^h_i(A(x,\na u))=\frac{\pa}{\pa x_i}\int_0^1|A(x+the_i,\na u(x+the_i))dt
\ \ [\mbox{Using (\ref{qp2})}]
\end{displaymath}
Say,
\begin{displaymath}
X=\int_0^1A (x+the_i,\na u(x+the_i))dt
\end{displaymath}

Therefore 2nd term of RHS of (\ref{a}) becomes
\begin{eqnarray}
&& -2\Iom \xi(D^h_i u)\langle D^h_i\langle A(x,\na u(x))\na\xi\rangle \no
\\
&&= 2\langle X,\frac{\pa}{\pa x_i}(D^h_iu.\xi\na\xi)\rangle\no
\\
&&\leq\frac{2}{R}\Iom\xi |X||D^h_i u_{x_i}|dx+\frac{C}{R^2}\int_{B_{2R}}|D^h_i u||X| dx \label{eqn5}
\end{eqnarray}

Therefore using the above relations, we have from (\ref{a})
\begin{eqnarray}
 C\Iom {\xi}^2|D^h_i(\na u(x))|^2 Y^{p-2}dx &\leq& C_2\Iom{\xi}^2(1+|\na
 u(x)|)^{p-1}|D_i^h\na u(x)|\no
\\
&+&\frac{2}{R}\Iom\xi |X| |D^h_i
 u_{x_i}|dx+\frac{C}{R^2}\int_{B_{2R}}|D^h_i u| |X| dx\no
\\
&-&\Iom\frac{u^{\ps-1}D^{-h}_i(\xi^2D^h_iu)}{|y|^t}.\label{b}
\end{eqnarray}
Now 1st term of RHS of (\ref{b}) can be estimated as follows
\begin{eqnarray*}
 C_2\Iom{\xi}^2(1+|\na u(x)|)^{p-1}|D_i^h\na u(x)|&\leq& \var\Iom {\xi}^2|D_i^h\na u(x)|^p
 dx\no
\\
&+& \frac{1}{C(\var)}\Iom{\xi}^2(1+|\na u(x)|)^p dx\no
\\
&\leq& \var\Iom {\xi}^2|D_i^h\na u(x)|^p dx+ M(\var)
\end{eqnarray*}

Again,
\begin{eqnarray*}
\frac{2}{R}\xi |X| |D^h_i u_{x_i}|&\leq&\frac{2}{R}\xi Y^{\frac{p-2}{2}}|D_i^h(\na u)|Y^{\frac{2-p}{2}}|X|
\\
&\leq& \var{\xi}^2Y^{p-2}|D_i^h\na u|^2+\frac{1}{\var R^2}Y^{2-p}|X|^2
\end{eqnarray*}
Therefore we have from (\ref{b})
\begin{eqnarray}
 (C-\var)\Iom\xi^2|D_i^h\na
 u|^2Y^{p-2}dx &\leq&
 \frac{1}{\var R^2}\int_{B_{2R}}Y^{2-p}|X|^2+\frac{C}{R^2}\int_{B_{2R}}|D_i^h
 u|X|dx\no
\\
&+& M(\var)+\var\Iom{\xi}^2|D_i^h \na u(x)|^p\no
\\
&-&\Iom\frac{u^{\ps-1}D^{-h}_i(\xi^2D^h_iu)}{|y|^t}\label{eqn6}
\end{eqnarray}
Now using following there inequalities (which follows by $ab\leq a^p+b^{p^\prime}$, where $p>1$ and $\frac{1}{p}+\frac{1}{p^\prime}=1$):
\begin{eqnarray*}
|D_i^h \na u|^p &\leq& Y^{p-2}|D_i^h \na u|^2+Y^p
\\
Y^{2-p}|X|^2 &\leq& Y^p+|X|^{\frac{p}{p-1}}
\\
|D_i^h u||X| &\leq& |D_i^h u|^p+|X|^{\frac{p}{p-1}},
\end{eqnarray*}
we get from (\ref{eqn6})
\begin{eqnarray}
 (C-2\var)\Iom\xi^2|D_i^h \na u|^p dx &\leq& C_1\int_{B_{2R}}Y^p
 dx+C_2\int_{B_{2R}}X^{\frac{p}{p-1}}dx\no
\\
&+& C_3\int_{B_{2R}}|D_i^h u|^p dx+M(\var)\no
\\
&-&\Iom\frac{u^{\ps-1}D^{-h}_i(\xi^2D^h_iu)}{|y|^t}\label{eqn7}
\end{eqnarray}
where the constants depend on $R$. Again we have,
\begin{displaymath}
 \int_{B_{2R}}Y^p\leq C_4 R^N+C_5 \int_{B_{3R}}|\na u|^p
\end{displaymath}
\begin{eqnarray}
\int_{B_{2R}}X^{\frac{p}{p-1}}dx &=& \int_{B_{2R}}\big(\int_0^1 |A(x+the_i,\na u(x+the_i))|dt\big)^\frac{p}{p-1}dx\no
\\
&\leq& \int_{B_{2R}}\int_0^1|A(x+the_i,\na u(x+the_i))|^\frac{p}{p-1}dtdx\no
\\
&\leq& C_1\int_{B_{2R}}\int_0^1|\na u(x+the_i)|^pdtdx \  \  [\mbox{using (\ref{qp2})}]\no
\\
&\leq& C_1\int_{B_{3R}}|\na u|^p dx
\end{eqnarray}
for $h$ to be small enough and $\int_{B_{2R}}|D_i^h u|^p dx\leq \int_{B_{3R}}|\na u|^p dx$.\\
Now the last integral in (\ref{eqn7}) can be estimated as follows:
\begin{eqnarray*}
&& \Iom\frac{u^{\ps-1}D^{-h}_i(\xi^2D^h_iu)}{|y|^t}
\\
&& \leq \var\Iom |D^{-h}_i(\xi^2D^h_iu)|^p+
C(\var)\Iom \Big(\frac{u^{\ps-1}}{|y|^t}\Big)^{p^\prime}\ [\mbox{Using Young's Inequality}]
\end{eqnarray*}
where $p^\prime=\frac{p}{p-1}$ and $C(\var)=\frac{(\var p)^{-\frac{p^\prime}{p}}}{p^\prime}$.
\begin{eqnarray*}
 \mbox{The 2nd term on the RHS} &\leq& 
M\Iom \frac{1}{|y|^{p^\prime t}}\ \ [\mbox{Since}\ u\in L^\infty(\Ome)]
\\
&\leq & K\ \ [\mbox{Since}\ p^\prime t<k],
\end{eqnarray*}
where $M,\ K$ are suitable constants depending on $\var$.
\begin{eqnarray*}
 \mbox{1st term on the RHS} &\leq & \var \int_{B_{3R}} |\na(\xi^2D^h_iu)|^p
\\
&\leq & C_1\var\Big[\int_{B_{3R}}\xi^2|D^h_i(\na u)|^p+\int_{B_{3R}}|D^h_iu|^p\Big]
\\
&\leq & C_1\var\Big[\int_{B_{3R}}\xi^2|D^h_i(\na u)|^p+\Iom|\na u|^p\Big],
\end{eqnarray*}
where $C_1$ is a constant depending on $R$. Therefore choosing $\var>0$ small enough
such that $2\var+C_1\var<C$ we have from (\ref{eqn7})
\begin{displaymath}
 \int_{B_R}|D^h_i\na u|^p\leq C(R,\var,N,p)\ \ \ \fa\ i=1,\ldots,N
\end{displaymath}
and for all sufficiently small $|h|\neq 0$. Therefore we have $D^h_iu$ is bounded in
$\Wuno\Rightarrow$ converges weakly and pointwise up to a subsequence to $u_{x_i}$.
Therefore by weak lower semicontinuity we have $\displaystyle\Iom|\na u_{x_i}|^p\leq M$
and this proves that $u\in W^{2,p}_{loc}(\Ome)$.

\spa

Now we come to the proof of Theorem \ref{W2p}. We note that if $A(x,\zeta)=|\zeta|^{p-2}\zeta$ then all the conditions in (\ref{qp1})-(\ref{qp3}) are
satisfied. Hence the solution $u$ of (\ref{A}), which is also in $L^\infty(\Ome)$
(see first paragraph of proof of Theorem \ref{C1al}), has local
$W^{2,p}$ regularity in the interior of $\Ome$. For the boundary points we first
flatten the boundary locally around $x_0\in\pa\Ome\setminus\{y=0\}$ using a $C^2$ deffiomorphism $\eta$ as in proof of Theorem \ref{C1al}. If $v=u\circ\eta^{-1}$
then $v$ satisfies
\begin{equation}
-\mbox{div}A(x,\na v)=g(x,v(x))\ \ \ \mbox{in}\ B^+_1(0)\label{diff}
\end{equation}
where $A(x,\zeta)=|\zeta B(x)|^{p-2}\zeta B(x)B^T(x)|\mbox{det}B(x)^{-1}|$, $B(x)=\nabla\eta(x), g\in L^\infty(B^+_1(0))$ and $B^T$ denotes the transpose of $B$
as we have seen in Theorem $3.3$. 
This operator satisfies
all the assumptions (\ref{qp1})-(\ref{qp3}) [see Appendix]. Using Leiberman \cite{L}
we have $u\in C^1(\pa\Ome\setminus\{y=0\})$ and so we can apply reflection method
to the above equation . If $x_0\in\pa\Ome\cap\{y=0\}$ we can apply the reflection method to the equation (\ref{diff}) (
using the same idea as in Theorem 2.3 of \cite{BS}), where reflection is given in the last coordinate of $x$.  Therefore
if we define $\bar v$ as in proof of Theorem \ref{C1al}, we note that $\bar v$ satisfies an equation
of type
\begin{displaymath}
-\mbox{div}\bar{A}(x,\na \bar v)=\frac{\bar{v}^{\ps-1}|\mbox{det}B(x)^{-1}|}{|y|^t}\ \ \mbox{in}\ \tilde\Ome,
\end{displaymath}
where $\bar{A}$ satisfies all the assumptions (\ref{qp1})-(\ref{qp3}). Therefore we have $\bar v\in W^{2,p}(\Ome^\prime), \eta(x_0) \in \Ome^\prime \Subset \tilde\Ome$,  which gives in particular boundary $W^{2,p}$ regularity near $x_0$. This completes the proof.
 
\end{proof}

\section{Non-existence of non trivial solutions in critical case}

\begin{teor} There does not exist any non trivial solution of the Equation (\ref{A}) in a star-shaped bounded domain w.r.t. $0$ for $t=s$ and
$s<k\big(\frac{p-1}{p}\big)$ if $\pa\Ome$ is orthogonal to the singular set.
\end{teor}
\begin{proof}
We will prove the theorem by contradiction. Therefore let us suppose $\Ome$ be a star shaped bounded domain and equation (\ref{A}) has a non trivial solution $u$. Hence by Theorem
\ref{W2p}, $u\in W^{2,p}(\Ome)$. Let $\va\in C^\infty(\R)$ be such that $\mbox{supp}(\va)
\subseteq [1,\infty)$, $0\leq\va\leq 1$ and $\va\equiv 1$ in $[2,\infty)$. Define
$\va_\var(x)\df\va(\frac{|y|}{\var})$. Therefore $\na\va_\var(x)=(\frac{y}{\var|y|}\va'(\frac{|y|}{\var}),0)$. Applying $\lan x.\na u\ran\va_\var\in W^{1,p}(\Ome)$ as a test function we have from (\ref{A})
\begin{eqnarray}
-\Iom\mbox{div}(|\na u|^{p-2}\na u)\lan x.\na u\ran\va_\var = \Iom \frac{u^{\ps-1}}{|y|^s}\lan x.\na u\ran\va_\var.\label{nonex} 
\end{eqnarray}
Now
\begin{eqnarray*}
 \mbox{RHS of (\ref{nonex})} &=& \frac{1}{\ps}\Iom\lan\na 
|u|^{\ps}.x\ran\frac{\va_\var}{|y|^s}
\\
&=& -\frac{1}{\ps}\sum_{i=1}^n\Iom |u|^{\ps}\frac{\pa}{\pa x_i}(\frac{x_i}{|y|^s})\va_\var
-\frac{1}{\ps}\Iom \frac{|u|^{\ps}}{|y|^s}\lan x.\na \va_\var\ran
\\
&& \mbox{[Since}\ u\in\Wuno\ \mbox{and}\ \frac{x}{|y|^s}\in W^{1,p}
(\Ome\setminus\{|y|\leq\var\})]
\\
&=& -\frac{N-p}{p}\Iom\frac{|u|^{\ps}}{|y|^s}\va_\var-\frac{1}{\ps}\Iom\frac{|u|^{\ps}}{|y|^s}
\lan x.\na\va_\var\ran
\\
&=& -\frac{N-p}{p}\Iom \frac{|u|^{\ps}}{|y|^s}\va_\var\ \mbox{as}\ \var\to 0.
\end{eqnarray*}

\begin{eqnarray*}
&& \mbox{LHS of (\ref{nonex})}
\\
 && =\Iom |\na|^{p-2}\na u \na(\lan x,\na u\ran\va_\var)-\Idom|\na u|^{p-2}\frac{\pa
 u}{\pa\nu}\lan x,\na u\ran\va_\var ds
\\
&&\ \ \ \ [\ \mbox{Since u is in}\ W^{2,p}(\Ome) \ \mbox{and $\nu$ is the normal vector to}\
\pa\Ome\ ]
\\
&& = \Iom |\na u|^p\va_\var + \frac{1}{2}\Iom|\na u|^{p-2}\langle\na|\na u|^2.x\rangle\va_\var+
\\
&&\ \Iom |\na u|^{p-2}\langle x,\na u\rangle\langle\na u,\na\va_\var
\rangle-\Idom|\na u|^{p-2}(\frac{\pa u}{\pa\nu})^2(x.\nu)\va_\var ds
\\
&&\ \ [\ \mbox{Applying the fact}\ x.\na u=\langle x,\nu\rangle\frac{\pa u}{\pa \nu}
\mbox{and}\ u\in C^1(\pa\Ome\setminus\{y=0\})]
\end{eqnarray*}
Now we note that the last term on the RHS tends to $\Idom |\na u|^p\langle x,\nu\rangle ds$
as $\var\to 0$. Similarly 1st term on the RHS $\rightarrow\Iom|u|^p$ as $\var\to 0$. So
we need to estimate the 2nd and 3rd terms of the RHS.
\begin{eqnarray*}
 \mbox{3rd term} &=& \Iom |\na u|^{p-2}\lan x,\na u\ran\lan\na u,\na\va_\var\ran
\\
&\leq & \int_{\Ome\cap\{\var\leq |y|\leq 2\var\}}|\na u|^{p-2}|x.\na u|\frac{|\na u||\va'|}{\var}
\\
&\leq & C \int_{\Ome\cap\{\var\leq |y|\leq 2\var\}} |\na u|^{p-1}\frac{|\na u|}{\var}
\\
&\leq & C_1 \int_{\Ome\cap\{\var\leq |y|\leq 2\var\}}|\na u|^{p-1}\frac{|\na u|}{|y|}
\\
&\leq & C_1(\Iom \frac{|\na u|^p}{|y|^p})^{\frac{1}{p}}(\int_{\Ome\cap\{\var\leq |y|\leq 2\var\}}|\na u|^p)^\frac{p-1}{p}\longrightarrow 0 
\end{eqnarray*}
as $\var\to 0$. To estimate the 2nd term, first we notice that
\begin{displaymath}
 \frac{1}{p}\lan x,\na(|\na u|^2)^\frac{p}{2}\ran = \frac{1}{2}\lan x, \na(|\na u|^2)\ran
|\na u|^{p-2}.
\end{displaymath}
Therefore
\begin{eqnarray*}
&& \frac{1}{2}\Iom \lan x,\na(|\na u|^2)\ran|\na u|^{p-2}\va_\var
\\
&& = -\frac{N}{p}\Iom |\na u|^p\va_\var + \frac{1}{p}\Idom\va_\var|\na u|^p\lan x,\nu\ran ds
-\frac{1}{p}\Iom \lan x,\na\va_\var\ran |\na u|^p
\\
&& \rightarrow -\frac{N}{p}\Iom |\na u|^p + \frac{1}{p}\Idom |\na u|^p\lan x,\nu\ran ds
\ \ \ \mbox{as}\ \var\to 0.
\end{eqnarray*}
Therefore letting $\var\to 0$, we have from (\ref{nonex})
\begin{eqnarray*}
&& -\frac{N-p}{p}\Iom\frac{u^{\ps}}{|y|^s}=\Iom |\na u|^p-\frac{N}{p}\Iom |\na u|^p
+(\frac{1}{p}-1)\Idom \lan x,\nu\ran |\na u|^p ds
\\
&&\Rightarrow \Idom \lan x,\nu\ran |\na u|^p ds=0 \ \ \ [\mbox{since u is a solution of (\ref{A})}]
\\
&& \Rightarrow \na u=0 \ \mbox{a.e. in}\ \pa\Ome.
\end{eqnarray*}
Since $u$ is $C^1$ on $\pa\Ome\setminus\{y=0\}$ we have $\na u=0$ on $\pa\Ome\setminus\{y=0\}$. But this a contradition to Hopf's lemma (see Theorem \ref{vaz})
since $s<k(\frac{p-1}{p})<\frac{k}{2}$.

\end{proof}

\spa

\section{Cylindrical symmetry and the set of degeneracy}

In this section we study some properties of the set of degeneracy of solutions of (\ref{A})  by the method of symmetry under the
condition $t<\frac{k(p-1)}{Np}$. We have already seen that there exists a strict positive solution $u$ of (\ref{A}) for $t<\frac{k(p-1)}{Np}$ and $\Ome$ connected. 
Therefore in this section we assume $t<\frac{k(p-1)}{Np}$ and $\Ome$ connected.

Intuitively, we can expect that $u$ could be symmetric in the variable $y$ only about the point $0\in\Rk$. 
Without loss of generality, we assume that $\Ome$ is a smooth bounded domain which
is cylindrically symmetric about $0$. Let $u\in C^1(\bar\Ome)
$ be a solution of (\ref{A}) which is strictly positive in $\Ome$. Before we state the results, let us define
some notations. Let $\nu$ be a direction in $\Rn$, i.e. $\nu\in\Rn$ and $|\nu|=1$. For
any real number $\la$ we define
\begin{eqnarray*}
 T^\nu_\la &\df & \{x\in\Rn  :\ x.\nu=\la\}
\\
\Ome^\nu_\la &\df & \{x\in\Ome :\ x.\nu<\la\}
\\ 
x^\nu_\la &\df & R^\nu_\la(x)=x+2(\la-x.\nu)\nu, \ \ \ x\in\Rn.
\end{eqnarray*}
Therefore $R^\nu_\la(x)$ is the reflection of $x$ through the hyperplane $T^\nu_\la$. Define
\begin{displaymath}
 a(\nu)\df\inf_{x\in\Ome}x.\nu.
\end{displaymath}
If $\la>a(\nu)$ then $\Ome^\nu_\la$ is non-empty and so we set
\begin{displaymath}
 (\Ome^\nu_\la)^\prime\df R^\nu_\la(\Ome^\nu_\la).
\end{displaymath}
$\Ome$ being smooth we observe that $(\Ome^\nu_\la)'$ is contained in $\Ome$ for 
$\la-a(\nu)>0$ small and will remain in it , at least until one of the following holds:
\begin{enumerate}
\item[(i)] $(\Ome^\nu_\la)'$ becomes internally tangent to $\pa\Ome$ at some point not
on $T^\nu_\la$.
\item[(ii)] $T^\nu_\la$ is orthogonal to $\pa\Ome$ at some point.
\end{enumerate}
Let $\Lambda_1(\nu)$ be the set of those $\la>a(\nu)$ such that for each $\mu\in(a(\nu),\la)$
none of the conditions (i) and (ii) holds and define
\begin{displaymath}
 \la_1(\nu)\df \sup \Lambda_1(\nu).
\end{displaymath}
Define
\begin{displaymath}
 \La_2(\nu)\df \{\la>a(\nu) :\ (\Onm)'\subset\Ome\ \mbox{for any}\ \mu\in(a(\nu),\la]\} 
\end{displaymath}
and if $\La_2(\nu)\neq\emptyset$ then define
\begin{displaymath}
 \la_2(\nu)\df \sup\La_2(\nu).
\end{displaymath}
Since $\Ome$ is smooth we have $\emptyset\neq\La_1(\nu)\subset \La_2(\nu)$ .
\\[1mm]
For $a(\nu)<\la\leq\la_2(\nu)$ we define
\begin{displaymath}
 u^\la_\nu(x)\df u(x^\la_\nu)\ \ \mbox{for}\ \ x\in\Onl.
\end{displaymath}
Since $u\in C^1(\Ome)$ we can also define
\begin{eqnarray*}
 Z^\nu_\la\df Z^\nu_\la(u) &=& \{x\in\Onl :\ Du(x)=Du^\nu_\la(x)=0\}
\\
Z\df Z(u) &=& \{x\in\Ome :\ Du(x)=0\}.
\end{eqnarray*}
Finally we define
\begin{displaymath}
 \La_0(\nu)\df\{\la\in(a(\nu),\la_2(\nu)]\ :\ u\leq u^\nu_\mu\ \mbox{in}\ \Onm\ \mbox{for
 any}\ \mu\in(a(\nu),\la] \}.
\end{displaymath}
If $\La_0(\nu)\neq\emptyset$ we set
\begin{displaymath}
 \la_0(\nu)=\sup\La_0(\nu).
\end{displaymath}
Obviously we have $\la_0(\nu)\leq \la_2(\nu)$.\\
The main Theorem of this section is the following:

\begin{teor}\label{singular}
Let $\Ome$ be as above with $\pa\Ome$ orthogonal to the singular set
and $\la_1(\nu)=\la_2(\nu)=0$ for all $\nu=(\nu_1, 0)\in\Rk\times\Rnk$ with $|\nu_1|=1$. Then
\begin{displaymath}
\mathcal{H}^{k-1}(Z)=0\Rightarrow u\ \mbox{is symmetric in variable}\ y\ \mbox{about $0$ and}\ Z\subset\{y=0\}.
\end{displaymath}
\end{teor}

\begin{rem}
It is easy to see that if $u$ is symmetric in variable $y$ about $0\in\Rk$ and $\mathcal{H}^{k-1}(Z)=0$ then $Z\subset\{y=0\}$.
Therefore it is enough to prove that $u$ is symmetric in variable $y$ about $0\in\Rk$ if $\mathcal{H}^{k-1}(Z)=0$.
\end{rem}

To prove the above theorem we crucially follow the method in \cite{DP} with suitable modification for our set up. Therefore we shall not provide
a detailed proof for the above theorem whereas we shall provide the steps involved to prove result by providing theorems
analogous those in \cite{DP}.
Before we proceed to prove the symmetry results, let us recall some theorems (valid for general open bounded domain $\Ome\subset\Rn$) which
we will use for our results.
\begin{teor}\label{vaz}(Strong Maximum Principle and Hopf's Lemma:)
Let $u\in C^1(\Ome)$ be such that $\Delta_pu\in L^2_{loc}(\Ome)$, $u\geq 0$ a.e. in $\Ome$,
$\Delta_pu\leq 0$ a.e. in $\Ome$. Then if $u$ does not vanish identically on $\Ome$ it is
positive everywhere in $\Ome$. Moreover, if $u\in C^1(\Ome\cup\{x_0\})$ for some $x_0\in\pa\Ome$ that satisfies an interior sphere condition and $u(x_0)=0$ then
\begin{displaymath}
 \frac{\pa u}{\pa n}>0,
\end{displaymath}
where $n$ is an interior normal at $x_0$.
\end{teor}
This theorem is a special case of Theorem 5 in \cite{V}.

\spa

\begin{teor}\label{har}(Harnack type comparison inequality:)\ Suppose $u,v$ satisfies
\begin{eqnarray}
 -\Delta_pu\leq-\Delta_pv\ \ \mbox{and}\ \ u\leq v\ \mbox{in}\ \Ome\label{har1}
\end{eqnarray}
where $u,v\in W^{1,\infty}_{loc}(\Ome)$. Suppose $\overline{B_{5\de}(x_0)}\subset\Ome$
and $\displaystyle\inf_{B_{5\de}(x_0)}(|Du|+|Dv|)>0$. Then for any positive $s<\frac{N}{N-2}$
we have
\begin{displaymath}
 \|v-u\|_{L^s(B_{2\de}(x_o))}\leq c\de^{\frac{N}{s}}\inf_{B_{\de}(x_0)}(v-u)
\end{displaymath}
where $c$ is constant depending on $N,p,s,\de$ and $m$ and $M$, where \\
$m=\displaystyle\inf_{B_{5\de}(x_0)}(|Du|+|Dv|)$, $M=\displaystyle\sup_{B_{5\de}(x_0)}(|Du|+|Dv|)$.
\end{teor}
It's a particular case of Theorem 1.3 in \cite{LD}. As a consequence of this theorem one
can have the following Strong Comparison Principle whose proof can be found in \cite{LD} 

\spa

\begin{teor}\label{SCP}
(Strong Comparison Principle:) Define $Z^v_u=\{x\in\Ome :\ Du(x)=Dv(x)=0\}$.
 Let $u,v\in C^1(\bar{\Ome})$ be such that it satisfy (\ref{har1})
and there exists $x_0\in\Ome\setminus Z^v_u$ with $u(x_0)=v(x_0)$,
then $u\equiv v$ in the connected component of $\Ome\setminus Z^v_u$ containing $x_0$.
\end{teor}

Now let us start with a technical lemma.
\begin{lemma}\label{YL}
 Let $\nu\in\Rn$ such that $\nu=(\nu_1,0)$ where $\nu_1\in\Rk$ and $|\nu_1|=1$. If $\la\leq0$
and $x^\nu_\la=(y^\nu_\la,z^\nu_\la)$ then $|y|\geq |y^\nu_\la|$ on $\Ome^\nu_\la$.
\end{lemma}

\begin{proof}
 We note that $y^\nu_\la=y+ 2(\la-y.\nu_1)\nu_1$. Therefore
\begin{eqnarray*}
 |y^\nu_\la|^2-|y|^2 &=&  4(\la-y.\nu_1)\nu_1.y + 4(\la-y.\nu_1)^2
\\
&=& 4(\la-y.\nu_1)[\nu_1.y+\la-y.\nu_1]
\\
&=& 4(\la-x.\nu)\la \leq 0.
\end{eqnarray*}

\end{proof}

\spa

Suppose $\Ome$ is a bounded domain in $\Rn$ with $\pa\Ome$ is orthogonal to the singular set. We know that for $\la\leq\la_2(\nu)$, $u\in C^1(\Onl)$ weakly solves
\begin{eqnarray}
 -\Delta_p u = \frac{u^{p^*(s)-1}}{|y|^t} \  \  \mbox {in}\ \Onl.\label{eqn11}
\end{eqnarray}
Clearly $\unl\in C^1(\Onl)$ weakly solves 
\begin{eqnarray}
 -\Delta_p \unl = \frac{{\unl}^{p^*(s)-1}}{|\ynl|^t} \  \   \mbox {in}\ \Onl.\label{eqn12}
\end{eqnarray}
where $\nu$ is as in Lemma \ref{YL}. For any set $A\subset\Onl$ we define
\begin{eqnarray*}
 M_A=M_A(u,\unl)=\sup_A(|Du|+|D{\unl}|)
\end{eqnarray*}
and we denote by $|A|$ its Lebesgue measure. Next we prove the following Weak Comparison Principle which is in the heart of the proof of Theorem \ref{singular}.

\spa

\begin{teor}\label{WCP}(Weak Comparison Principle:)
Suppose that $1<p<2$ and $\nu=(\nu_1,0)$, then for any $\la\leq 0$ there exists
$\de_1,\ \de_2>0$, depending on $p,\ |\Ome|,\ M_\infty$($M_\infty=2\sup_\Omega|Du|$) and the $L^\infty$ norm of $u$ such that:
if an open set $\Ome'\subset \Onl$ satisfies $\Ome'=A_1\cup A_2, |A_1\cap A_2|=0, |A_1|<\de_1,
M_{A_2}<\de_2$ then $u\leq\unl$ on $\pa\Ome'$ implies $u\leq\unl$ in $\Ome'$.
\end{teor}

\begin{proof}
Let $u\leq\unl$ on $\pa\Ome'$. Therefore $(\unl-u)^-=0$ on $\pa\Ome'$ and hence
$(\unl-u)^-\in W^{1,p}_0(\Ome')$. Define 
\begin{displaymath}
A(x)\df \frac{{\unl}^{\ps-1}-u^{\ps-1}}{\unl-u}.
\end{displaymath}

Using $(\unl-u)^-$ as a test function we have from 
(\ref{eqn11}) and (\ref{eqn12})
\begin{eqnarray*}
 &&\int_{\Ome'}(|\na\unl|^{p-2}\na\unl)-|\na u|^{p-2}\na u).\na(\unl-u)^-
\\
&&\geq\int_{\Ome'}A(x)\frac{(\unl-u)(\unl-u)^-}{|y|^t},\ \mbox{[using Lemma \ref{YL}]}
\end{eqnarray*}
implying
\begin{eqnarray*}
&&\int_{\Ome'\cap\{\unl\leq u\}}(|\na u|^{p-2}\na u-|\na\unl|^{p-2}\na\unl)).\na(\unl-u)^-
\\
&&\leq \int_{\Ome'\cap\{\unl\leq u\}}A(x)\frac{|(\unl-u)^-|^2}{|y|^t},
\end{eqnarray*}
and hence
\begin{eqnarray}
 \int_{\Ome'\cap\{\unl\leq u\}}(|\na\unl|+|\na u|)^{p-2}|\na(\unl-u)^-|^2\leq
 M\int_{\Ome'\cap\{\unl\leq u\}}\frac{|(\unl-u)^-|^2}{|y|^t}\label{eqn13}
\end{eqnarray}
where $M$ denotes the upper bound of $A(\cdot)$ and so depends on $\|u\|_{L^\infty(\Ome)}$.
Since $1<p<2$ we have from (\ref{eqn13})
\begin{eqnarray}
&& M_\infty^{p-2}\int_{A_1\cap\{\unl\leq u\}}|\na(\unl-u)|^2+M_{A_2}^{p-2}\int_{A_2\cap\{\unl\leq u\}}
|\na(\unl-u)|^2\no
\\
&& \leq\frac{M}{\var^t}\int_{\Ome'\cap\{\unl\leq u\}}|(\unl-u)^-|^2+
\int_{\Ome'\cap\{\unl\leq u\}\cap\{|y|<\var\}}\frac{|(\unl-u)^-|^2|y|^{2-t}}{|y|^2}.\label{eqn14}
\end{eqnarray}
Since $u\in C^1(\bar\Ome)$ we have $(\unl-u)^-\in W^{1,2}_0(\Ome')$, therefore following Lemma 2.2 in \cite{LD} and Hardy Sobolev Inequality we have
\begin{eqnarray*}
 \mbox{RHS of (\ref{eqn14})} &\leq &\frac{2M}{\var^t}\frac{|\Ome|^{1/N}}{|\sigma_N|^{2/N}}
\Big[|A_1|^{1/N}|\na(\unl-u)|^2_{L^2(A_1\cap\{\unl\leq u\})}+
\\
&& |\Ome|^{1/N}|\na(\unl- u)|^2_{L^2(A_2\cap\{\unl\leq u\})}\Big]
\\
&&+C\var^{2-t}\int_{A_1\cap\{\unl\leq u\}}|\na(\unl-u)|^2
\\
&&+C\var^{2-t}\int_{A_2\cap\{\unl\leq u\}}|\na(\unl-u)|^2
\end{eqnarray*}
where $\sigma_N$ denotes the unit sphere in $\Rn$.
Therefore from (\ref{eqn14}) we get 
\begin{eqnarray}\label{eqn15}
&& \Big[M_{\infty}^{p-2}-\frac{2M|\Ome|^{1/N}|A_1|^{1/N}}{\var^t|\sigma_N|^{2/N}}-C\var^{2-t}\Big]
\int_{A_1\cap\{\unl\leq u\}}|\na(\unl-u)|^2 \no
\\
&& + \Big[M_{A_2}^{p-2}-\frac{2M|\Ome|^{2/N}}{\var^t|\sigma_N|^{2/N}}-C\var^{2-t}\Big]
\int_{A_2\cap\{\unl\leq u\}}|\na(\unl-u)|^2\leq 0.
\end{eqnarray}
Now choosing $C\var^{2-t}=\frac{M_\infty^{p-2}}{2}$, the term in the first bracket of (\ref{eqn15})
becomes
\begin{displaymath}
\frac{M_{\infty}^{p-2}}{2} -\frac{2M|\Ome|^{1/N}|A_1|^{1/N}}{C_1|\sigma_N|^{2/N}},
\end{displaymath}
with $C_1=\var^t=(\frac{M_\infty^{p-2}}{2C})^\frac{t}{2-t}$
and hence for $|A_1|\leq\de_1$, $\de_1>0$ to be sufficiently small, the above quantity
is positive. For the same choice of $\var$ the term in the second bracket of (\ref{eqn15})
becomes
\begin{displaymath}
 \frac{1}{M_{A_2}^{2-p}}-\Big(\frac{2M|\Ome|^{2/N}}{C_1|\sigma_N|^{2/N}}
+\frac{M_\infty^{p-2}}{2}\Big)
\end{displaymath}
and therefore we can choose $\de_2>0$ such that if $M_{A_2}<\de_2$ then above quantity
is positive. Therefore we have
\begin{displaymath}
 \int_{A_i\cap\{\unl\leq u\}}|\na(\unl-u)|^2= 0\ \ \mbox{for}\ i=1,2.
\end{displaymath}
This implies 
\begin{displaymath}
 \|(\unl-u)^-\|_{W^{1,2}_0(\Ome')}=0
\end{displaymath}
and hence $u\leq\unl$ in $\Ome'$. This completes the proof.
\end{proof}

Now we are in the position to prove Theorem \ref{singular}.
 Since the main idea of the proof comes from \cite{DP}
with some suitable modification for our set up, we would just provide a sketch for
the proof of Theorem \ref{singular}.
\\[2mm]

\noi{\bf Sketch of the proof of Theorem \ref{singular}:} \\
First we prove results analogous to Theorem 3.1 in \cite{DP} and then complete the proof of
Theorem \ref{singular} above along the lines at page 700--705 in \cite{DP}. 
To proceed first we'll show that $\la_0(\nu)=\la_2(\nu)$ for all $\nu=(\nu_1, 0)$ with $|\nu_1|=1$ which will prove $u\leq u_{\la_2(\nu)}^\nu$ in $\Ome_{\la_2(\nu)}^\nu$.
To get other side inequality observe that $v(x)\df u(x_{\la_2(\nu)}^\nu)$ satisfies an equation of same type in $\Ome$. 

\noi Step 1:\ Let $\nu=(\nu_1,0)$ and $\la_2(\nu)=0$. Since $u=0$ on $\pa\Ome$ and
$u=\unl$ on $\Ome\cap T^\nu_\la$, applying Theorem \ref{WCP} we have that $u\leq\unl$
in $\Onl$ for $\la-a(\nu)>0$ small. Therefore $\La_0(\nu)\neq\emptyset$. At this
point we observe that $u\leq u^\nu_{\la_0(\nu)}$ in $\Ome^\nu_{\la_0(\nu)}$, by continuity. Therefore 
\begin{displaymath}
 -\Delta_p u^\nu_{\la_0(\nu)}=\dfrac{{u^\nu_{\la_0(\nu)}}^{\ps-1}}{|(y)_{\la_0(\nu)}|^t}
\geq \dfrac{{u}^{\ps-1}}{|y|^t}=-\Delta_p u\ \ \mbox{in}\ \ \Ome^\nu_{\la_0(\nu)}.
\end{displaymath}
[Here we have used Lemma \ref{YL}].
So we can apply Theorem \ref{SCP}. Then following
the same arguments as in Step 2 of Theorem 3.1 in \cite{DP} we have, if $\la_0(\nu)<\la_2(\nu)$
then there exists at least one connected component $C^\nu$ of $\Ome^\nu_{\la_0(\nu)}
\setminus Z^\nu_{\la_0(\nu)}$ such that $u=u^\nu_{\la_0(\nu)}$ in $C^\nu$. Again
it can also be shown that for any $\la$ with $a(\nu)<\la<\la_0(\nu)$ we have
\begin{displaymath}
 u<\unl \ \ \mbox{in}\ \Onl\setminus Z^\nu_\la.
\end{displaymath}

\noi Step 2. Now we shall prove that
\begin{displaymath}
\dfrac{\pa u}{\pa \nu}(x)>0 \ \ \fa\ x\in \ \Ome^\nu_{\la_0(\nu)}\setminus Z
\end{displaymath}
which will prove the non-increasing property of $u$.
For simplicity we assume $\nu=e_1=(1,\ldots,0)$. Let $x=(\la,\varsigma)\in\Ome^\nu_{\la_0(\nu)}\setminus Z$, i.e., $\la<\la_0(\nu)$ and
$Du(x)\neq 0$. Therefore we can choose a ball $B_r(x)$ such that $|Du|,\ |D\unl|\geq\var>0$
in $B_r(x)\cap\Onl$ for some $\var$ small enough and $B_r(x)\cap\{y=0\}=\emptyset$. Then
by standard results $u\in C^2(B_r(x))$ and the difference $\unl-u$ satisfies a linear
strictly elliptic equation $L(\unl-u)=0$ (cf. \cite{S}). On the other hand we have 
$\unl-u>0$ in $B_r(x)\cap\Onl$ while $\unl=u$ on $T^\nu_\la$. Hence, by Hopf's lemma
(cf. \cite{S}) we get $0>\frac{\pa(\unl-u)}{\pa x_1}(x)=-2\frac{\pa u}{\pa x_1}(x)$ i.e.,
$\frac{\pa u}{\pa x_1}>0$.
\\[2mm]
\noi Step 3. Till now
we have not used the fact that $\mathcal{H}^{k-1}(Z)=0$. To complete the proof it is enough to prove that $\la_0(\nu)=\la_2(\nu)$.
We shall prove this by contradiction. Therefore assume
$\la_0(\nu)<\la_2(\nu)$. Let
\begin{displaymath}
\bar{I}_{\de}(\nu)\df\{\mu\in\Rn : \ \mu=(\mu_1,0), \mu_1\in\Rk, |\mu_1|=1, |\mu_1-\nu_1|<\de\}.
\end{displaymath}
Therefore $\bar{I}_\de(\nu)$ is an $(k-1)$-dimensional open set embedded in the $(N-1)$-dimensional sphere. Using Theorem \ref{WCP}
and arguments at page 702--703 in \cite{DP} we can show that $\la_0(\cdot)$ is continuous on $\bar{I}_\de(\nu)$. Therefore following
the arguments similar to those at page 703--705 in \cite{DP} we can have a set $Z_1\subset Z$(in fact, on $Z_1$, $u$=const., $Du=0$) such that 
$\mathcal{H}^{k-1}(Z_1)>0$. This is contradicting to our hypothesis. This completes the proof of Theorem \ref{singular}. \hfill $\Box$

\spa

Note that we can also prove the weak maximum principle (Theorem \ref{WCP}) for the directions of type $\nu=(0,\nu_2)$ with $\nu_2\in\Rnk$ with 
$|\nu_2|=1$. Therefore following the same arguments as above we have the following theorem.

\begin{teor}
Let $\Ome$ be as above with $\pa\Ome$ orthogonal to the singular set
and $\la_1(\nu)=\la_2(\nu)=0$ for all $\nu=(0, \nu_2)\in\Rk\times\Rnk$ with $|\nu_2|=1$. Then
\begin{displaymath}
\mathcal{H}^{N-k-1}(Z)=0\Rightarrow u\ \mbox{is symmetric in variable}\ z\ \mbox{about $0$ and}\ Z\subset\{z=0\}.
\end{displaymath}
\end{teor} 
Here if $N-k>1$ it goes back to the previous argument but if  $N-k=1$  by definition of Hausdroff measure we can say that $Z=\emptyset$. Now from Step 1 above
we can conclude $u=u_{\la_0(\nu)}^\nu$ in $\Omega_{\la_0(\nu)}^\nu$ and by continuity $u=u_{\la_0(\nu)}^\nu$ in $\bar\Omega_{\la_0(\nu)}^\nu$. Now to show $\la_0(\nu)=\la_2(\nu)$ let us observe that if $\la_0(\nu)<\la_2(\nu)$ it contradicts the fact $u=u_{\la_0(\nu)}^\nu$ in $\bar\Omega_{\la_0(\nu)}^\nu\cap\pa\Ome$ since we have $u>0$ in $\Ome$.

As an immediate consequence of the above theorems we have
\begin{cor} 
Let $\Ome$ be as above with $\pa\Ome$ orthogonal to the singular set and $\ell=\min\{k-1,N-k-1\}$. Then
\begin{displaymath}
\mathcal{H}^{\ell}(Z)=0\Rightarrow u\ \mbox{is cylindrically symmetric about $0$ and}\   Z=\{0\}.
\end{displaymath}
\end{cor}

\section{Appendix}

Let $x_0$ be a point on $\pa\Ome$. Since $\Ome$ is a smooth bounded domain, we can find a $C^2$ diffeomorphism $f:\R^{N-1}\to \R$ s.t. after re-labeling and re-orienting the coordinate axes (if necessary) we have 
\begin{displaymath}
 \Ome_1=\Ome\cap B(x_0,r)=\{x\in B(x_0,r):x_n>f(x_1,\ldots,x_{n-1})\}
\end{displaymath}
Now let us define, $\eta:\Ome_1\to B_1^+ $ be corresponding deffiomorphism where $B_1^+=\eta(\Ome_1)$. And $\va(x)=\eta^{-1}(x)$. If $u$ satisfies the following equation
\begin{displaymath}
-div|\na u|^{p-2}\na u=g(x,u)
\end{displaymath}
in weak sense, then after the change of variable defined above with $v(x)=u(\va(x))$ where $v:B^+_1\to \R$, the above equation changes to
\begin{displaymath}
 A(x,\na v)=|\na v\cdot B(x)|^{p-2}\na v.B(x)B(x)^T|\mbox{det} B(x)^{-1}|
\end{displaymath}
where $B(x)=D\eta(x)$. Therefore denoting $\na v$ as $\xi$ we get
\begin{displaymath}
 A(x,\xi)=|\xi\cdot B(x)|^{p-2}\xi.B(x)B(x)^T|\mbox{det} B(x)^{-1}|
\end{displaymath}
We claim that $A(x,\xi)$ satisfies (\ref{qp1})-(\ref{qp3}).

\spa

Proof of (\ref{qp1}): We can see that
\begin{eqnarray*}
 &&\big(A(x, \xi_1)-A(x, \xi_2)\big)(\xi_1-\xi_2)\\
&&\geq \big(|\xi_1 B(x)|^{p-2}\xi_1 B(x)-|\xi_2
 B(x)|^{p-2}\xi_2 B(x)\big)(\xi_1 B(x)-\xi_2 B(x))|\mbox{det} B(x)^{-1}|\\
&&\geq C_1(1+|\xi_1|^2+|\xi_2|^2)^\frac{p-2}{2}|\xi_1 B(x)-\xi_2 B(x)|^2.
\end{eqnarray*}
Again
\begin{eqnarray*}
 |\xi_1 B(x)-\xi_2 B(x)|^2 &=&\langle \xi_1.B(x)B(x)^T-\xi_2.B(x)B(x)^T, \xi_1-\xi_2\rangle\\
&\geq& C_2 |\xi_1-\xi_2|^2
\end{eqnarray*}
Since $B(x)B(x)^T$ is a positive definite matrix. Hence (\ref{qp1}) follows.

\spa

Proof of (\ref{qp2}) is obvious. 

Proof of (\ref{qp3}):\begin{eqnarray*}
 A_i(x,\xi)-A_i(y,\xi)&=&|\xi B(x)|^{p-2}\big[(\xi B(x)B(x)^T)_i-(\xi
 B(y)B(y)^T)_i\big]|\mbox{det} B(x)^{-1}|\no
\\
&+&(\xi B(y)B(y)^T)_i|\xi B(x)|^{p-2}|\mbox{det} B(x)^{-1}|\no
\\
&-&(\xi B(y)B(y)^T)_i|\xi B(y)|^{p-2}|\mbox{det} B(y)^{-1}|\no
\\
&\leq& C|x-y||\xi|^{p-1}+C|\xi|\big[|\xi B(x)|^{p-2}-|\xi B(y)|^{p-2}\big]|\mbox{det} B(x)^{-1}|\no
\\
&+& (\xi B(y)B(y)^T)_i|\mbox{det} B(x)^{-1}||\xi B(y)|^{p-2}\no
\\
&-&(\xi B(y)B(y)^T)_i|\mbox{det} B(y)^{-1}||\xi B(y)|^{p-2}
\end{eqnarray*}
Since \begin{displaymath}
 |\xi B(x)|^{p-2}-|\xi B(y)|^{p-2}\leq C|\xi|^{p-2}|x-y|
\end{displaymath}
 and $\eta$ is a $C^2$ function in  a bounded domain implies $|\mbox{det} B(y)^{-1}|$ is a Lipschitz function. Therefore result follows.\hfill $\Box$

\noi {\bf Acknowledgement:}\ This paper is a part of first author's doctoral dissertation.
The first author would like to thank her adviser K. Sandeep for various useful discussion.

\end{document}